\documentclass[fullpage]{article}

\usepackage{amsmath,amsthm,amsfonts,amssymb,amscd,epsf,epsfig,psfrag,enumerate,bm,tikz,bbold}
\usepackage[linesnumbered]{algorithm2e}

\usepackage{hyperref} 

\newtheorem{theorem}{Theorem}
\newtheorem{lemma}[theorem]{Lemma}

\newcommand{\ve}{\boldsymbol}
\newcommand{\Cr}{\operatorname{CR}}
\newcommand{\aCr}{\operatorname{CR^{\rm a}}}

\DeclareMathOperator{\rank}{rank}

\DeclareMathOperator{\lin}{lin}

\DeclareMathOperator{\pos}{pos}
\DeclareMathOperator{\supp}{supp}
\DeclareMathOperator{\intt}{int}

\DeclareMathOperator{\aff}{aff}

\DeclareMathOperator{\GL}{GL}

\newcommand{\R}{\mathbb{R}}
\newcommand{\Z}{\mathbb{Z}}

\newcommand{\bA}{\bm{A}}
\newcommand{\bx}{\bm{x}}

\newcommand{\bh}{\bm{h}}
\newcommand{\bB}{\bm{B}}
\newcommand{\ba}{\bm{a}}
\newcommand{\br}{\bm{r}}
\newcommand{\be}{\bm{e}}

\newcommand{\by}{\bm{y}}
\newcommand{\bz}{\bm{z}}

\newcommand{\bb}{\bm{b}}

\newcommand{\bv}{\bm{v}}
\newcommand{\bw}{\bm{w}}
\newcommand{\bU}{\bm{U}}

\newcommand{\bH}{\bm{H}}

\newcommand{\BIGOP}[1]{\mathop{\mathchoice%
{\raise-0.22em\hbox{\huge $#1$}}%
{\raise-0.05em\hbox{\Large $#1$}}{\hbox{\large $#1$}}{#1}}}

\newcommand{\BIGboxplus}{\mathop{\mathchoice%
{\raise-0.35em\hbox{\huge $\boxplus$}}%
{\raise-0.15em\hbox{\Large $\boxplus$}}{\hbox{\large $\boxplus$}}{\boxplus}}}

\usepackage{authblk}
\title{New Bounds for the Integer Carath\'{e}odory Rank}
\author[1]{Iskander Aliev}
\author[2]{Martin Henk}
\author[1]{Mark Hogan}
\author[2]{Stefan Kuhlmann}
\author[3]{Timm Oertel}
\affil[1]{Cardiff University, United Kingdom}
\affil[2]{Technische Universit\"{a}t Berlin, Germany}
\affil[3]{Friedrich-Alexander-Universität Erlangen-Nünrberg, Germany}
\date{}                     
\begin{document}
	\maketitle
	
\noindent \textbf{Abstract.} 
Given a rational pointed $n$-dimensional cone $C$, we study the integer Carath\'{e}odory rank $\Cr(C)$ and its asymptotic form $\aCr(C)$, where we consider ``most'' integer vectors in the cone. The main result significantly  improves the previously known upper bound for $\aCr(C)$. 
We also study bounds on $\Cr(C)$ in terms of $\Delta$, the maximal absolute $n\times n$ minor of the matrix given in an integral polyhedral representation of $C$. If $\Delta\in\lbrace 1,2\rbrace$, we show $\Cr(C) = n$, and prove upper bounds for simplicial cones, improving the best known upper bound on $\Cr(C)$ for $\Delta\leq n$.

\section{Introduction}
A cone $C$ in $\R^n$ is  {\em rational} if there exists an integer ${m\times n}$ matrix $\bA$ such that 
	\begin{align}
		\label{set_rational_cone}
		C= \lbrace\bx\in \R^n : \bA\bx\geq\bm{0}\rbrace,
	\end{align}
	where ${\ve 0}$ is the zero vector and the inequality is componentwise.
The cone $C$ is {\em pointed } if $C\cap(-C)=\{{\ve 0}\}$, that is the origin is the vertex of $C$, or equivalently, $\bA$ has full column rank. The {\em dimension} of the cone $C$ is the cardinality of a maximal set of linearly independent vectors in $C$. 

Given a finite set $G=\{{\ve g}_1, \ldots, {\ve g}_t\}\subset \Z^n$,  the \emph{semigroup} $S$ with {\em generating set}  $G$ is defined as 
\begin{equation}\label{semigroup}
S = \{\lambda_1 {\ve g}_1+ \cdots +\lambda_t {\ve g}_t:  \lambda_1, \ldots, \lambda_t \in \Z_{\ge 0} \}.
\end{equation}

Let $C\subset \R^n$ be a rational pointed $n$-dimensional cone. The integer points in the cone $C$ form a semigroup $S=C\cap\Z^n$. Due to a result of van der Corput \cite{vandercorputhilberbasisunique1931} (see also \cite{Jeroslow78} and \cite{Schrijver81}), the semigroup $S$ has a uniquely determined inclusion--minimal finite generating set $H=H(C)$. The set $H$ has been traditionally referred to as the \emph{Hilbert basis} of $C$. In the theory of mathematical optimisation, Hilbert bases are strongly related to  {\em Totally Dual Integral (TDI)-systems} ~\cite[Chapter 22.3]{schrijvertheorylinint86}, and  {\em Graver bases} \cite{Aardal_Weismantel_Wolsey,Graver1975}. 

A classical theorem by Cara\-th\'eodory states that each point of the cone $C$ is a non-negative combination of at most $n$ vectors which lie on extreme rays of $C$.
Cook, Fonlupt, and Schrijver posed in~\cite{CookFS1986} the following question, analogous to the one answered by Cara\-th\'eodory's theorem:

\noindent
{\em -- What is the smallest $k$ such that every integer point of the cone $C$ can be expressed as a non-negative integer combination of at most $k$ vectors in the Hilbert basis~$H$?}

%

{
We will refer to $k$ as the \emph{integer Carath\'eodory rank} of $C$, and denote it by $\Cr(C)$. 
To formally define this quantity, we consider for any element ${\ve x}$ of the  semigroup $C\cap\Z^n$ its {\em representation length}
\begin{equation*}
\sigma({\ve x})=\min\{l: {\ve x}=\lambda_1{\ve h}_1+\cdots+\lambda_l{\ve h}_l, \lambda_i\in \Z_{\ge 0}, {\ve h}_i\in H(C) \}\,.
\end{equation*}
Then 
\begin{equation*}
\Cr(C)=\max\{\sigma({\ve x}): {\ve x}\in C \cap \Z^n \}\,.
\end{equation*}
}

%

%

The paper~\cite{CookFS1986} gave the upper bound $\Cr(C) \le 2n-1$ which was, subsequently, applied in the context of  {\em TDI-systems}, {\em integer rounding property} of integer programs, {\em independent sets} of \emph{matroids}, 
and {\em coverings} of {\em perfect graphs}. The current best known upper bound 
\begin{equation}\label{Sebo}
\Cr(C) \le 2n-2
\end{equation}
 was obtained by Seb\H{o} in~\cite{sebohilbertbasisdreidim90}. We also remark that the notion of the integer Carath\'eodory rank was extended by Eisenbrand and Shmonin \cite{eisenbrandshmonincaratheodorybounds06} to general semigroups $S$ of the form \eqref{semigroup}. 
See also \cite{alievAverkovLoeraOertel21,alievdeloesparelindio2017} and references within.

Following the work of Bruns and Gubeladze \cite{brunsgubeladzenormalpoly1999}, we say that a pointed rational $n$-dimensional cone $C\subset \R^n$
satisfies the \emph{Integral Carath\'eodory Property} (ICP) if $\Cr(C) = n$.
It was conjectured in \cite{sebohilbertbasisdreidim90} that the ICP holds for every $n$-dimensional cone $C$. This conjecture was disproved by Bruns et al. in \cite{brunsgubehenkcounterexampleintcara99}. Specifically, it was shown in \cite{brunsgubehenkcounterexampleintcara99} that in every dimension $n\ge 6$ there exists an $n$-dimensional cone $C$ with $\Cr(C) \ge \left \lfloor7n/6 \right\rfloor$.

{
To study the ``typical'' maximal representation length, Bruns and Gubeladze introduced in \cite{brunsgubeladzenormalpoly1999} the \emph{asymptotic integer Carath\'eodory rank $\aCr(C)$} of the cone $C$, which is defined as the smallest positive integer $k$ such that the following limit exists
and satisfies the equality
\begin{equation*}
\lim_{\delta\rightarrow \infty}\frac{|\{{\ve x}\in S: \sigma({\ve x})\le k\}\cap [-\delta, \delta]^n |}{|S\cap [-\delta, \delta]^n |}=1\,.
\end{equation*}
That is ``most'' vectors in $C \cap \Z^n$ can be represented by at most $k$ Hilbert basis elements.

}

Clearly, $\aCr(C)\le \Cr(C)$. It was shown in \cite{brunsgubeladzenormalpoly1999} that 
\begin{equation}\label{BG_asymptotic}
\aCr(C) \le 2n-3\,
\end{equation}
and that in every dimension $n\geq 6$ there exists an $n$-dimensional cone $C$  with $\aCr(C)>n$.

Known results on the integer Carath\'eodory rank lead to two interesting and long--standing open questions:

\noindent
{\em -- What are the optimal upper bounds for $\Cr(C)$ and  $\aCr(C)$ in terms of $n$?}
In the case of  $\Cr(C)$, Seb\H{o}'s bound \eqref{Sebo}  remains the best known upper estimate for over three decades. Further, the work of Gubeladze \cite{gubeladzesurveynormal2023} indicates that reducing \eqref{Sebo} to a bound of the form $\Cr(C)\le (2-\epsilon)n$ with $\epsilon>0$ for all sufficiently large $n$ is a challenging problem. In particular, it would disprove a conjecture (\cite[Conjecture 2.1]{gubeladzesurveynormal2023}) on the integer Carath\'eodory rank of normal polytopes.

In this paper, we study the above question in the case of the asymptotic integer Carath\'eodory rank. 
Theorem \ref{asymUpBound} reduces the bound \eqref{BG_asymptotic} to $\aCr(C) \le \lfloor 3n/2\rfloor$. On the other hand, Theorem \ref{asymLowBound} shows that in every dimension $n\ge 6$ there exists an $n$-dimensional cone $C$ with $\aCr(C) \ge \lfloor 7n/6\rfloor$. 

\noindent
{\em -- What cones have the integer Carath\'eodory property?} 
Cook, Fonlupt, and Schrijver ~\cite{CookFS1986} observed that the ICP holds for two-dimensional cones. Subsequently, Seb\H{o} \cite{sebohilbertbasisdreidim90} proved the ICP for cones of dimension three. On the other hand, due to the result of Bruns et al. \cite{brunsgubehenkcounterexampleintcara99}, there exist cones that do not satisfy the ICP for every $n\geq 6$. Despite this, it remains an active line of research to classify which cones admit the ICP; see \cite{pinamatroidcararank2003,gijswijtipdicp2012} for some results concerning cones related to matroids. We continue this line of research by investigating the integer Carath\'eodory rank in terms of the parameter
\begin{align*}
	\Delta(\bA) = \max \left\lbrace \left|\det \bB\right| : \bB \text{ is an } n\times n \text{ submatrix of }\bA\right\rbrace,
\end{align*}
where $\bA\in\Z^{m\times n}$ has full column rank. We refer to $\bA$ as \textit{$\Delta$-modular} if $\Delta(\bA)=\Delta$. 

Recently, significant effort has been made to understand the computational complexity of integer programming problems defined by $\Delta$-modular matrices. Three key results in this area are given in \cite{artmannweiszen17,Fiorini2022IntegerPW,naegelesanzencongruence2022}. This task motivated the study of polyhedral geometry depending on the parameter $\Delta(\bA)$; see \cite{bonisummaeisenbranddiameterpoly14,celayakuhlpaarweis22,celayakuhlpaatweis2022proxandflatness,HKW2022} for an incomplete collection of results concerning the distance of optimal integral solutions of an integer linear program and optimal vertex solutions of the corresponding relaxation, the lattice width of lattice-free polyhedra, and the diameter of polyhedra. In some of the recent advancements, in particular \cite{celayakuhlpaatweis2022proxandflatness,HKW2022}, Hilbert bases play a central role when proving novel upper bounds which solely depend on $\Delta(\bA)$.

An intriguing special case are \emph{simplicial cones}, that means cones, where $\bA\in\Z^{n\times n}$ in (\ref{set_rational_cone}) satisfies $\det\bA\neq 0$. Even for simplicial cones it is open whether they admit the ICP. An affirmative answer to this has, combined with some extra effort, the following strong implication: the integer vectors contained in the zonotope spanned by the primitive generators of a non-simplicial cone have the ICP. 
In addition to this, the study of simplicial cones and the ICP relates to various other concepts in mathematics such as simplices with the integer decomposition property which themselves are connected to weighted projective spaces; see for instance \cite{braun2018detecting,conrads2002weighted}.

In this paper, we consider the above question of which cones admit the ICP independently from the dimension of the cone. Theorem \ref{thm_outer_unimod_bimod} shows that the ICP holds in arbitrary dimension for cones with $\Delta(\bA)\in\lbrace 1,2\rbrace$. We further strengthen this for simplicial cones and obtain an improvement on the bound \eqref{Sebo} if $\Delta(\bA)\leq n$; see Theorem \ref{thm_main_simplicial_polyhedral}.

	In what follows, by $\intt X$ we denote the interior of a set $X$, $\lin X$ is the linear hull of $X$, $\pos X$ is the positive hull of $X$, and $\aff X$ is the affine hull of $X$. 
%
%
We use the notation $\lbrack m\rbrack$ for the set $\lbrace 1,\ldots,m\rbrace$. Given $\bA\in\Z^{m\times n}$, $I\subseteq\lbrack m \rbrack$, and $J\subseteq \lbrack n\rbrack$, we denote by $\bA_{I,J}$  the submatrix of $\bA$ with rows indexed by $I$ and columns indexed by $J$. If $J=\lbrack n\rbrack$, we write $\bA_{I,\cdot}$ and similarly $\bA_{\cdot,J}$ when $I = \lbrack m\rbrack$.
%
In the same manner, given a vector $\bx\in\R^n$ and a set $I\subseteq\lbrack n\rbrack$, we denote by $\bx_I\in\R^{|I|}$ the vector with coordinates indexed by $I$. The \textit{support of $\bx$} is defined as 
	$
		\supp(\bx) = \lbrace i\in\lbrack n\rbrack : \bx_{\{i\}} \neq 0\rbrace.
	$	
	We denote by $\GL(n,\Z)$ the group of all $n\times n$ unimodular matrices, that is $\bA \in\Z^{n \times n}$ and $\left|\det \bA\right| = 1$. The standard unit vectors in $\R^n$ are denoted by $\be_1,\ldots,\be_n$.

\section{Statement of results}\label{Results}

Our main result strengthens the bound \eqref{BG_asymptotic} obtained by Bruns and Gubeladze  \cite{brunsgubeladzenormalpoly1999}. 

\begin{theorem}\label{asymUpBound}
Let $C$ be a rational pointed $n$-dimensional cone.
Then
\[
\aCr(C)\le\left\lfloor\tfrac{3}{2}n\right\rfloor.
\]
\end{theorem}

The second result gives a new lower bound for the maximal value of the asymptotic integer Carath\'eodory rank of an $n$-dimensional cone.
\begin{theorem}\label{asymLowBound}
For every integer $n\ge 6$ there exists a rational pointed $n$-dimensional cone $C_n$ such that
\begin{equation}\label{asympt_lower_bound}
\aCr(C_n)\ge\left\lfloor\tfrac{7}{6}n\right\rfloor.
\end{equation}
\end{theorem}

%

%
To state our parameterized results, we consider a rational pointed cone
	\begin{align*}
		C(\bA) = \lbrace\bx\in \R^n : \bA\bx\geq\bm{0}\rbrace\,
	\end{align*}
and estimate its integer Carath\'{e}odory rank in terms of the parameter $\Delta(\bA)$.

	Firstly, we show that the ICP holds for the cone $C(\bA)$ if all $n\times n$ subdeterminants of $\bA$ are bounded by two.
	\begin{theorem}
		\label{thm_outer_unimod_bimod}
		Let $\bA\in\Z^{m\times n}$ be a matrix of full column rank with $\Delta(\bA)\le 2$. Then $\Cr(C(\bA))=n$.
	\end{theorem}
	Note that a counterexample for Seb\H{o}'s conjecture obtained in \cite{brunsgubehenkcounterexampleintcara99} has a polyhedral representation with $\Delta(\bA) = 144$. The smallest $2<\Delta(\bA)\le 144$ for which the ICP fails is not known. 
		
	Suppose now  that $\bA\in \Z^{n\times n}$ is a nonsingular matrix.  Then $C(\bA)$ is a simplicial cone and $ \Delta(\bA)=\left|\det\bA\right|$. In this setting, we obtain an upper bound for $\Cr(C(\bA))$ which combines the parameter $\Delta(\bA)$ with the dimension $n$. This results in an improvement on the bound \eqref{Sebo} for $\Delta(\bA)\leq n$.

	\begin{theorem}\label{thm_main_simplicial_polyhedral}
		Let $\bA\in\Z^{n\times n}$ be a nonsingular matrix.
		\begin{itemize}
			\item[(i)] If $1\leq \Delta(\bA) \leq 4$, then $\Cr(C(\bA))=n$.
			\item[(ii)] If $\Delta(\bA)\ge 5$, then $\Cr(C(\bA))\leq n + \Delta(\bA) - 3$. 
		\end{itemize}
		
	\end{theorem}

\section{Proofs of Theorem~\ref{asymUpBound} and~\ref{asymLowBound}}

\subsection{Proof of Theorem~\ref{asymUpBound}}
{To prove the upper bound for the asymptotic integer Carath\'eodory rank $\aCr(C)$, it is sufficient to  construct a set $D \subseteq C\cap\Z^n$ that satisfies the following two properties.
\begin{enumerate}
 \item[(i)] For any point ${\ve b}\in D$ we have  $\sigma({\ve b})\le 3n/2$.
 \item[(ii)] We have
\[
\lim_{\delta \to \infty} \frac{\left| D\cap[-\delta,\delta]^n\right|}{\left| C \cap \Z^n\cap[-\delta,\delta]^n\right|} = 1\,.
\]
 \end{enumerate}
%

Let us take the Hilbert basis $H(C)=\{\bh_1,\ldots,\bh_t\}$ of the cone $C$ and consider the matrix  $\bH\in\Z^{n\times t}$ with columns $\bh_1,\ldots,\bh_t$.
Let further
\begin{equation*}
\Delta=\max\left\{\left|\det\bH_{\cdot,I}\right|\;:\; I\subseteq[ t],\; |I|=n\right\}.
\end{equation*}
We consider the set
\[
D=C\cap\Z^n \setminus\bigcup_{\tau\in{[ t]\choose n-1}}\left\{\sum_{i=1}^t\lambda_i\bh_i \;:\;\begin{array}{l} 0\le\lambda_i \,\quad\quad\text{ for all }i\in\tau,\\0\le\lambda_i<\Delta\text{ for all }i\in[t]\setminus\tau \end{array}\right\}.
\]
Note that for each $\tau\in{[ t]\choose n-1}$ we can write
\[
\left\{\sum_{i=1}^t\lambda_i\bh_i \;:\;\begin{array}{l} 0\le\lambda_i \,\quad\quad\text{ for all }i\in\tau,\\0\le\lambda_i<\Delta\text{ for all }i\in[t]\setminus\tau \end{array}\right\} = P + C_\tau,
\]
where 
\[
P=\bH\cdot[0,\Delta)^t=\{ \bH{\ve x}: {\ve x}\in   [0,\Delta)^t  \}
\]
and  
\[
C_\tau=\left\{\sum_{i\in\tau}\lambda_i\bh_i \;:\; \lambda_i\ge 0 \right\}\,.
\]
Observe that $\left| C \cap \Z^n \cap[-\delta,\delta]^n\right| \in\Theta(\delta^n)$. 
Similarly, for any $\tau\in{[ t]\choose n-1}$ it holds that $\left| (P+C_\tau)\cap\Z^n \cap[-\delta,\delta]^n \right|\in \Theta(\delta^{k})$, where $k=\rank \bH_{\cdot,\tau}\le |\tau| < n$.
It holds
\[ \frac{\left| D \cap[-\delta,\delta]^n\right|}{\left| C \cap \Z^n\cap[-\delta,\delta]^n\right|} \ge  1-  \frac{
\sum_{\tau\in{[ t]\choose n-1}} \left|(P+C_\tau)\cap\Z^n \cap[-\delta,\delta]^n\right| }{\left| C \cap \Z^n\cap[-\delta,\delta]^n\right|} .
\]
The latter ratio tends to zero, as $\delta$ tends to infinity.
Hence, $D$ satisfies (ii).

Given  $\ve b \in \Z^n $, we will use the notation $Q(\bH, {\ve b})$ for the polyhedron 
\begin{equation*}
Q(\bH, {\ve b})=\{{\ve x}\in \R^t_{\ge 0}: \bH{\ve x}={\ve b}\}.
\end{equation*}
To show that $D$ satisfies (i), it is sufficient to prove that for any $\bb\in D$, there exists a $\ve x\in Q(\bH, {\ve b})\cap\Z^t$ with $|\supp(\ve x)|\le 3n/2$.
Let us consider the following linear optimisation problem:}
\begin{equation}\label{Norm_maximisation}
\max\left \{x_1+\cdots+x_t: 
\bx=(x_1,\ldots,x_t)^\top\in Q(\bH, \bb) \right\}\,.
\end{equation}
Since $C$ is pointed and $\bb\in C$, it is clear that \eqref{Norm_maximisation} is feasible and bounded.
Let ${\ve \lambda}$ be an optimal vertex solution for \eqref{Norm_maximisation}. Hence, ${\ve \lambda}$ has at most $n$ non-zero entries and, renumbering the coordinates, we may assume that $\lambda_{n+1}=\ldots=\lambda_t=0$.
Furthermore,  the condition $\bb\in D$ implies that $\lambda_1,\ldots,\lambda_n\ge\Delta$.

Note that $\bH_{\cdot, [n]}$ is an optimal basis for \eqref{Norm_maximisation}.
A classical result from the theory of linear programming implies that any vector $\ve \alpha = (\alpha_1,\ldots,\alpha_n,0,\ldots,0)^\top\in\R^t_{\geq 0}$ with $\bH \ve \alpha\in \Z^n$ is an optimal solution to the linear optimisation problem $\max_{\bx\in Q(\bH, \bH \ve \alpha)} (x_1+\cdots+x_t)$, see for example~\cite[Section 5.1]{bertsimas-LPbook}.

Let $\ve\mu=(\lambda_1-\lfloor \lambda_1 \rfloor,\ldots,\lambda_n-\lfloor \lambda_n \rfloor)^\top$.
%
Consider the vector  ${\ve r}=\bH_{\cdot,[n]}\ve\mu\in C\cap\Z^n$. We can write ${\ve r}= \sum_{i=1}^ 
t\beta_i\bh_i$ with $\beta_i\in\Z_{\ge 0}$.
Hence,
\[
\bb=\sum_{i=1}^n \lfloor\lambda_i\rfloor \bh_i +\sum_{i=1}^t\beta_i\bh_i.
\]

Observe that ${\ve \beta}=(\beta_1, \ldots, \beta_t)^\top\in Q(\bH, \br)$ by construction and that the linear optimisation problem 
\[
\max\{x_1+\cdots+x_t: 
\bx=(x_1,\ldots,x_t)^\top\in Q(\bH, \br)\}
\]
has an optimal vertex solution $(\mu_1, \ldots, \mu_n, 0, \ldots, 0)^\top$. 
Hence,
$\sum_{i=1}^t\beta_i\le\sum_{i=1}^n\mu_i$.

If $\sum_{i=1}^n\mu_i\le n/2$, then at most $\lfloor n/2\rfloor$ of the numbers $\beta_i$ can be non-zero and the result follows.
To settle the case $\sum_{i=1}^n\mu_i > n/2$, we consider the vector ${\ve \gamma}= (\lceil\lambda_1\rceil-\lambda_1,\ldots,\lceil\lambda_n\rceil-\lambda_n)^\top$.
We have ${\ve s}=\bH_{\cdot,[n]}{\ve \gamma}\in C\cap\Z^n$ and, consequently,  ${\ve s}=\sum_{i=1}^t\delta_i\bh_i$ with  $\delta_i\in\Z_{\ge 0}$. 
Observe that ${\ve \delta}=(\delta_1, \ldots, \delta_t)^\top\in Q(\bH, {\ve s})$ and that the linear optimisation problem 
\[
\max\{x_1+\cdots+x_t: 
\bx=(x_1,\ldots,x_t)^\top\in Q(\bH, {\ve s})\}
\]
has an optimal vertex solution $(\gamma_1, \ldots, \gamma_n, 0, \ldots, 0)^\top$. 
Hence, 
\[
\sum_{i=1}^t\delta_i\le\sum_{i=1}^n\gamma_i< n/2\,
\]
and, consequently, ${\ve s}$ can be expressed as non-negative integer combination of strictly less than $n/2$ Hilbert basis elements.
Let $q=\left|\det\bH_{\cdot,[n]}\right|\le\Delta$.
Then $q{\ve \gamma}$ is integral by Cramer's rule.
Recall that we have $\bb \in D$ and, in particular, $\lambda_i\ge\Delta$ for all $i\in[n]$.
Thus, we obtain that
\begin{equation*}
\eta_i=\lambda_i-(q-1)\gamma_i= (\lambda_i + \gamma_i)-q \gamma_i
\end{equation*}
are non-negative integers for all $i\in[n]$. 
Finally, we can express $\bb$ as 
\begin{equation*}
\begin{aligned}
{\ve b} &= \sum_{i=1}^n \lambda_i \bh_i = \sum_{i=1}^n (\eta_i+(q-1)\gamma_i)\bh_i\\
&=\sum_{i=1}^n \eta_i \bh_i  + (q-1){\ve s}\,.
\end{aligned}
\end{equation*}
This completes the proof since ${\ve s}$ is the non-negative integral combination of strictly less than $n/2$ Hilbert basis elements.
\qed

\subsection{Proof of Theorem~\ref{asymLowBound}} Let $C$ be a pointed rational $n$-dimensional cone in $\R^n$.
Theorem 6.1 in \cite{brunsgubeladzenormalpoly1999} implies that if $\aCr(C)=n$, then $\Cr(C)=n$.
The counterexample to the integer Carath\'eodory conjecture, shown in \cite{brunsgubehenkcounterexampleintcara99}, provides a 6-dimensional cone $C_6$ with 
$\Cr(C_6)=7$. Hence, by Theorem 6.1, we get the lower bound $\aCr(C_6)>6$.  Furthermore, the inequality $\aCr(C_6)\le\Cr(C_6)$ implies $\aCr(C_6)=7$.

%
Lemma 4.4 in \cite{brunsgubeladzenormalpoly1999} shows that $\aCr(C\times C')=\aCr(C)+\aCr(C')$.
Following the construction in \cite{brunsgubehenkcounterexampleintcara99} and setting
$
C_n=\left(\BIGOP{\times}_{i=1}^{\lfloor n/6 \rfloor}C_6\right)\times C',
$
where $C'$ is any pointed, full-dimensional, rational cone in $\R^{n\bmod 6}$, we obtain a cone that satisfies~\eqref{asympt_lower_bound}.\qed

	\section{Proofs of Theorems \ref{thm_outer_unimod_bimod} and \ref{thm_main_simplicial_polyhedral}} 
	\label{section_half_space_proofs}
 Throughout this section, we work with the polytope
	\begin{align*}
		P_{\bm{1}}(\bA) = \lbrace \bx\in\R^n : \bm{0}\leq \bA\bx\leq\bm{1}\rbrace,
	\end{align*}
	where $\bm{1}$ denotes the all-ones vector. 
	Note that $P_{\bm{1}}(\bA)$ is full-dimensional if and only if $C(\bA)$ is full-dimensional. Further, $P_{\bm{1}}(\bA)$ is bounded as $\bA$ has full column rank. 
	
	When proving Theorem \ref{thm_outer_unimod_bimod} and Theorem \ref{thm_main_simplicial_polyhedral}, we employ the following strategy: 
	Firstly, we argue that $P_{\bm{1}}(\bA)\cap\Z^n\backslash\lbrace\bm{0}\rbrace\neq\emptyset$. Then, given $\bz\in\intt C(\bA)\cap\Z^n$, this implies the existence of some Hilbert basis element $\bh\in H( C(\bA))$ such that the point $\bz - \lambda\bh$ for some $\lambda\in\Z_{>0}$ is contained in the boundary of $C(\bA)$; see Lemma \ref{lemma_outer_walk_to_face}. Next, we work with our new integer vector $\bz - \lambda\bh$ and iterate this procedure using Lemma \ref{lemma_outer_fulldim_polyhedra} below. Hence, we use at every step exactly one Hilbert basis element and the dimension of the face of $C(\bA)$ which contains the current integer vector in the relative interior decreases by at least one. Consequently, our strategy results in expressing $\bz$ as an integer combination of at most $n$ Hilbert basis elements. 
	
	We begin by proving the first step, that is, the existence of a Hilbert basis element $\bh$ from above. By doing so, we exploit a crucial property of Hilbert basis elements: given $\bh\in H(C)$ and $\by_1,\by_2\in C \cap \Z^n$ such that $\bh = \by_1 + \by_2$, then either $\by_1 = \bm{0}$ or $\by_2 = \bm{0}$; see, e.g., \cite[Chapter 16.4]{schrijvertheorylinint86} for some details.
	\begin{lemma}\label{lemma_outer_walk_to_face}
		Let $\bA\in\Z^{m\times n}$ be a full column rank matrix with rows ${\ve a}_1, \ldots, {\ve a}_m$ such that $P_{\bm{1}}(\bA)\cap\Z^n\backslash\lbrace\bm{0}\rbrace\neq \emptyset$. Given $\bz\in \intt C(\bA)\cap \Z^n$, there exists a Hilbert basis element $\bh \in H(C(\bA))$ and $\lambda\in\Z_{>0}$ such that $\bz - \lambda\bh\in C(\bA)$ and $\ba_i^\top(\bz - \lambda\bh) = 0$ for some $i\in\lbrack m\rbrack$.
	\end{lemma}
	\begin{proof}
		Let $\bh\in P_{\bm{1}}(\bA)\cap\Z^n\backslash\lbrace\bm{0}\rbrace$ be chosen such that $\left|\supp(\bA\bh)\right|$ is minimal among all vectors in $P_{\bm{1}}(\bA)\cap\Z^n\backslash\lbrace\bm{0}\rbrace$. We observe that
		\begin{align*}
			\lambda = \min_{i\in\supp(\bA\bh)} \ba_i^\top\bz
		\end{align*}
		already yields the claim for $\bh$ as $\bA\bh\in\lbrace 0,1\rbrace^m$. So it suffices to argue that $\bh$ is a Hilbert basis element.

		Let $\by_1,\by_2\in C(\bA)\cap\Z^n$ be such that $\bh = \by_1 + \by_2$. It is sufficient to show that one of the vectors $\by_1$, $\by_2$ is zero. Since $\bA\bh\in\lbrace 0,1\rbrace^m$, we have $\bA\by_i\in\lbrace 0,1\rbrace^m$ for $i=1,2$ as well. However, this implies that $\supp(\bA\by_i)\subseteq\supp(\bA\bh)$ for $i=1,2$. The minimality of $\left|\supp(\bA\bh)\right|$ implies that either $\by_1= \bm{0}$ or $\by_2=\bm{0}$.
		\qed
	\end{proof}
	Lemma \ref{lemma_outer_walk_to_face} enables us to argue inductively over the dimension $n$. To make this precise, we establish in the next result a representation for lower-dimensional faces and their minors. Let  $\bA\in\Z^{m\times n}$ be a matrix with full column rank and let $\bb\in\Z^m$.
In what follows, we consider the polyhedron $P(\bA,\bb)= \lbrace \bx\in\R^n : \bA\bx\leq\bb\rbrace$ and define 
	\begin{align*}
		\gcd(\bA) = \gcd(\det \bA_{I,\cdot} : I\subseteq\lbrack m\rbrack \text{ with } |I|= n)\,.
	\end{align*}
	For $\bA\in\Z^{m\times n}$ with full row rank, we set $\gcd(\bA)=\gcd(\bA^\top)$.
	
	\begin{lemma}
		\label{lemma_outer_fulldim_polyhedra}
		Let  $\bA\in\Z^{m\times n}$ be a matrix with full column rank and $\bb\in\Z^m$. Further, let $F_I= P(\bA,\bb)\cap \left(\bv + \ker \bA_{I,\cdot}\right)$ be a  $(n-k)$-dimensional face of $P(\bA,\bb)$ with $\aff F_I\cap\Z^n \neq \emptyset$, where $I\subseteq\lbrack m\rbrack$ with $\left| I \right| = k$ and $\bA_{I,\cdot}\bv = \bb_I$ hold. Then, there exists a unimodular transformation $\bU\in \GL(n,\Z)$ and orthogonal projection $\pi:\R^n\to\R^{n-k}$ with the following properties:
		\begin{itemize}
			\item[(i)] $\pi\left(\bU \cdot F_I\right)$ is a $(n-k)$-dimensional polyhedron that admits a representation of the form $\pi\left(\bU \cdot F_I\right)=P(\tilde\bA,\tilde\bb)$ with an integer vector $\tilde\bb$ and integer matrix $\tilde\bA $ which is at most $\left\lfloor\frac{\Delta(\bA)}{\gcd( \bA_{I,\cdot})}\right\rfloor$-modular.
			\item[(ii)] There exists a one-to-one mapping between $F_I\cap \Z^n$ and $\pi\left(\bU \cdot F_I\right)\cap\Z^{n-k}$.
		\end{itemize}
	\end{lemma}
	\begin{proof}
		We assume without loss of generality that $I = \lbrace 1,\ldots,k\rbrace$.
		There exists a unimodular transformation $\bU\in \GL(n,\Z)$, such that
		\begin{align*}
			\bA_{I,\cdot}\bU^{-1} = (
			\bH,\bm{0})
		\end{align*}
		for some invertible matrix $\bH\in\Z^{k\times k}$ that can be obtained, for instance, by transforming $\bA_{I,\cdot}$ into Hermite normal form; see \cite[Chapter 4]{schrijvertheorylinint86} for more information on the Hermite normal form. Moreover, we have
		\begin{align*}
			\bA\bU^{-1} = \begin{pmatrix}
				\bH & \bm{0} \\
				\star & \tilde{\bA}
			\end{pmatrix}
		\end{align*}
		for some $\tilde{\bA}\in \Z^{(m-k)\times (n-k)}$ with full column rank.
	 	
	 	Let $\pi : \R^n\to\R^{n-k}$ denote the orthogonal projection onto the last $n-k$ coordinates and let $\tilde{\bz}\in\R^k$ be the unique solution of $\bH\bx = \bb_I$. Then,
	 	\begin{align}\label{proof_lemma_outer_minor_projection}
	 		\pi\left(\bU \cdot F_I\right) = \left\lbrace \bx\in\R^{n-k} : \tilde{\bA}\bx\leq\bb_{\lbrack m \rbrack\backslash I} - \bA_{\lbrack m \rbrack\backslash I,\lbrack k \rbrack}\tilde{\bz}\right\rbrace
	 	\end{align}
	 	which is a $(n-k)$-dimensional polyhedron defined by the integral constraint matrix $\tilde{\bA}$. Let us show that $\tilde{\bA}$ is at most $\left\lfloor\frac{\Delta(\bA)}{\gcd \bA_{I,\cdot}}\right\rfloor$-modular. For that purpose, let $\bB$ be a $(n-k)\times (n-k)$ submatrix of $\tilde{\bA}$. We can extend the matrix to 
	 	\begin{align*}
	 		\begin{pmatrix}
	 			\bH & \bm{0} \\
	 			\star & \bB
	 		\end{pmatrix}
	 	\end{align*}
 		which is an $n\times n$ submatrix of $\bA\bU^{-1}$ with determinant $\left|\det \bB\right|\left|\det\bH\right|\leq \Delta(\bA)$. We have $\left|\det \bH\right| = \gcd\bA_{I,\cdot}$ which follows, e.g., from the Smith normal form; see for instance \cite[Chapter 4.4]{schrijvertheorylinint86} for a treatment of Smith normal forms. The latter equality and the integrality of $\tilde{\bA}$ imply that $\tilde{\bA}$ is at most $\left\lfloor\frac{\Delta(\bA)}{\gcd \bA_{I,\cdot}}\right\rfloor$-modular. Further, the right-hand side of the system in (\ref{proof_lemma_outer_minor_projection}) defining $\pi(\bU \cdot F_I)$ is given by $\bb_{\lbrack m \rbrack\backslash I} - \bA_{\lbrack m \rbrack\backslash I,\lbrack k \rbrack}\tilde{\bz}$. To settle property (i), it remains to prove the integrality of the right-hand side. We claim that $\tilde{\bz}$ has to be integral. The integrality follows then from the previous mentioned description of the right-hand side. 
 		
 		Since $\bU$ is unimodular, $\aff F_I\cap\Z^n\neq\emptyset$ implies $\aff (\bU \cdot F_I)\cap \Z^n\neq\emptyset$. Recall that $\tilde{\bz}\in\R^k$ denotes the unique solution of $\bH\bx = \bb_I$. Let $\bz\in\R^n$ be the vector $\tilde{\bz}$ with $n-k$ zeros appended. We have
 		\begin{equation*}
		\begin{aligned}
 			\aff (\bU \cdot F_I) &= \bz + \ker \bA_{I,\cdot}\bU^{-1} = \bz + \left\lbrace \bx\in \R^n : \bx_{\lbrack k \rbrack} 
			= \bm{0} \right\rbrace \\ &= \left\lbrace \bx\in \R^n : \bx_{\lbrack k \rbrack} = \tilde{\bz}\right\rbrace.
 		\end{aligned}
		\end{equation*}
 		Thus, $\aff (\bU \cdot F_I)\cap \Z^n\neq\emptyset$ implies $\tilde{\bz}\in\Z^k$. Hence, property (i) follows by the discussion above. Moreover, we obtain $\by\in \bU \cdot F_I\cap\Z^n$ if and only if $\pi(\by)\in \pi(\bU \cdot F_I)\cap \Z^{n-k}$, which settles property (ii).
 	\qed	
 			\end{proof}
%

\subsection{Proof of Theorem \ref{thm_outer_unimod_bimod}}
In this proof, we refer to $\bA$ as \textit{unimodular} and \textit{bimodular} if $\Delta(\bA) = 1$ and $\Delta(\bA) = 2$, respectively. 

Suppose first that $\Delta(\bA) = 1$. 
		We argue inductively on $n$. For $n = 1$, the cone is a ray and, thus, the statement immediately holds. Let $n > 1$ and $\bz\in C(\bA)\cap\Z^n$. If $\bz$ lies on the boundary of $C(\bA)$, we restrict to the face which contains $\bz$ in the relative interior and apply Lemma \ref{lemma_outer_fulldim_polyhedra} with respect to that face. This results in a lower-dimensional cone with unimodular constraint matrix and the statement of the theorem follows by induction. 
		
We may now assume $\bz\in\intt C(\bA)\cap\Z^n$, which implies that $C(\bA)$ is a full-dimensional cone. Since $P_{\bm{1}}(\bA)$ is defined by a unimodular matrix with integral right-hand side, every vertex of $P_{\bm{1}}(\bA)$ is integral. There are at least two vertices as $n\geq 2$. Hence, the polytope $P_{\bm{1}}(\bA)$ contains a non-zero integral vector. Using Lemmas \ref{lemma_outer_walk_to_face}  and \ref{lemma_outer_fulldim_polyhedra}, the statement of the theorem follows by induction.

Suppose now that $\Delta(\bA) = 2$. In this case, our proof crucially relies on an integer feasibility result due to Veselov and Chirkov:
\begin{theorem}{\cite[Theorem 1]{veselovchirkovbimodular09}}\label{thm_bimod_integer_feasibility}
	Let $\bA\in\Z^{m\times n}$ be bimodular and $\bb\in\Z^m$ such that $P(\bA,\bb)$ is full-dimensional. Then $P(\bA,\bb)\cap\Z^n\neq\emptyset$.
\end{theorem} 
	
		We argue again inductively over $n$. Similarly to the $\Delta(\bA) = 1$ case, we may assume that $n > 1$ and $\bz\in \intt C(\bA)\cap\Z^n$. Furthermore, we assume that every row of $\bA$ defines a facet of $C(\bA)$. If this is not the case, we remove rows of $\bA$ which do not correspond to a facet of $C(\bA)$. This operation does not increase $\Delta(\bA)$.
		
		As before, our first goal is to show that $P_{\bm{1}}(\bA)$ contains  a non-zero integer vector. Let $F_{\ba} =\lbrace \bx\in P_{\bm{1}}(\bA) : \ba^\top\bx = 1\rbrace$ define a facet of $P_{\bm{1}}(\bA)$ for some row $\ba$ of $\bA$.  In what follows, we distinguish between the cases $\gcd(\ba) = 1$ and $\gcd(\ba) = 2$. Note that $\gcd(\ba) \geq 3$ is not possible as it violates the assumption that $\bA$ is bimodular.
	
	Let $\gcd(\ba) = 1$. This implies that the affine hull of the facet $F_{\ba} $ contains integer vectors. Applying Lemma \ref{lemma_outer_fulldim_polyhedra}, we receive a full-dimensional polytope in $(n-1)$-dimensional ambient space which is bimodular. Hence, it contains an integer vector by Theorem \ref{thm_bimod_integer_feasibility}. Since the right-hand side equals one, this vector cannot be $\bm{0}$.
		
		Let $\gcd(\ba) = 2$. Every row of $\bA$ is facet defining for $C(\bA)$ as we assumed in the beginning of the proof. Therefore, there exists a facet of $P_{\bm{1}}(\bA)$ defined by $\ba^\top\bx = 0$. By Lemma \ref{lemma_outer_fulldim_polyhedra}, the facet defined by $\ba$ corresponds to a polytope with unimodular constraint matrix and integral right-hand side. As the linear hull of the facet contains integral vectors, every non-zero vertex is integral. There are at least two vertices since $n\geq 2$. So there exists a non-zero integer vector in the facet.

Using Lemmas \ref{lemma_outer_walk_to_face} and \ref{lemma_outer_fulldim_polyhedra}, the statement of the theorem follows by induction.
\qed

\subsection{Proof of Theorem \ref{thm_main_simplicial_polyhedral}}
 Recall that $\bA\in\Z^{n\times n}$ is a nonsingular matrix and, thus, $\Delta(\bA) = \left|\det\bA\right|$. In this manner, we write $\left|\det\bA\right|$ instead of $\Delta(\bA)$ throughout the proof. We will need the auxiliary result below. The proof is based on the theory of lattices; see \cite{GruLek87} for an introduction to lattices.
\begin{lemma}\label{former_claim}
		Suppose that $n\geq \left|\det\bA\right|$. Then the parallelepiped $P_{\bm{1}}(\bA) = \lbrace \bx\in\R^n : \bm{0}\leq \bA\bx\leq \bm{1}\rbrace$ contains a non-zero integer vector.
\end{lemma}	
\begin{proof}	
		Suppose that the matrix $\bA^{-1}$ has columns $\bw_1,\ldots,\bw_n$ and let $\Lambda = \bA^{-1}\Z^n$. Then $\bw_i\in\Lambda$ for each $i\in\lbrack n\rbrack$. Note that $\bA\bw_i = \be_i$ and hence $\bw_1,\bw_1+\bw_2,\ldots,\bw_1+\cdots+\bw_n\in P_{\bm{1}}(\bA)\cap \Lambda$. We analyse these sums with respect to the cosets of the finite abelian group $\Lambda / \Z^n$. Note that $\left| \Lambda / \Z^n\right| = \left|\det\bA\right|$ since $\det\Lambda = \left|\det \bA\right|^{-1}$. As $n\geq \left|\det \bA\right|$, either one of the sums is integral or two sums, say $\bw_1+\cdots+\bw_p$ and $\bw_1+\cdots+\bw_q$ for $p < q$, are contained in the same coset of $\Lambda / \Z^n$ by the pigeonhole principle. This implies that $\bw_{p+1}+\cdots+\bw_{q}\in P_{\bm{1}}(\bA)\cap\Z^n$. \qed
\end{proof}

%

		Suppose first that $1\leq \left|\det\bA\right| \leq 4$ and take any integer point ${\ve z}$ in $C(\bA)$. By Lemma \ref{lemma_outer_fulldim_polyhedra}, we may assume that $C(\bA)$ is full-dimensional and that $\bz\in \intt C(\bA)\cap \Z^n$. Since every cone of dimension at most three has the ICP \cite[Theorem 2.2]{sebohilbertbasisdreidim90}, we suppose that $n\geq 4$. Thus, we get $n\geq 4 \geq \left|\det\bA\right|$. 
		As $n\geq \left|\det\bA\right|$, Lemma \ref{former_claim} gives $P_{\bm{1}}(\bA)\cap\Z^n\backslash\lbrace\bm{0}\rbrace\neq\emptyset$.  Applying Lemmas \ref{lemma_outer_walk_to_face} and \ref{lemma_outer_fulldim_polyhedra} we replace ${\ve z}$ with an integer point in a lower-dimensional cone whose constraint matrix is at most $\left|\det \bA\right|$-modular. We repeat this procedure until the dimension is at most $\left|\det \bA\right| - 1$. For each iteration, we use exactly one Hilbert basis element and the number of iterations is at most $n - (\left|\det \bA\right| - 1)$. Next, we apply again the result of \cite{sebohilbertbasisdreidim90} stating that the ICP holds for cones in dimension $3 \geq \left|\det\bA\right| - 1$. Thus, we obtain an expression of ${\ve z}$ as an integer combination of at most $n$ elements of $H(C(\bA))$.
		
		Suppose now that $\left|\det\bA\right|\ge 5$ and take any integer point ${\ve z}$ in $C(\bA)$.  Observe that the case $n-1\leq\left|\det\bA\right|$ follows from (\ref{Sebo}). Hence, we may assume $n\geq\left|\det\bA\right|$.  Therefore, using Lemma \ref{former_claim} and then Lemmas   \ref{lemma_outer_walk_to_face} and \ref{lemma_outer_fulldim_polyhedra}  as above, we can replace ${\ve z}$ with an integer point in a lower-dimensional cone which is at most $\left|\det \bA\right|$-modular. We repeat this procedure until the dimension is at most $\left|\det \bA\right| - 1$. As $\left|\det \bA\right| - 1\geq 2$, we can apply Seb\H{o}'s bound, (\ref{Sebo}), and get at most $2\left(\left|\det\bA\right| - 1\right) - 2$ Hilbert basis elements in an integral combination. Together with our previous steps, which give us at most $n - (\left|\det\bA\right| - 1)$ Hilbert basis elements in an integral combination, we obtain an expression of ${\ve z}$ as a non-negative integer combination of at most
		\begin{align*}
			2\left(\left|\det\bA\right| - 1\right) - 2 + \left(n - \left(\left|\det \bA\right| - 1\right)\right) = n + \left|\det \bA\right| - 3
		\end{align*}
elements of $H(C(\bA))$.\qed
	
${}$\newline
\noindent{\bf Remark:}  We highlight the limitation of our approach: Given an integer vector ${\ve z}\in C(\bA)$, the strategy of searching for an element $\bh\in H(C(\bA))$ such that for some integer 
$\lambda$ the vector ${\ve z}-\lambda\bh$ reaches a lower-dimensional face of the cone $C(\bA)$ is limited to the case $\Delta(\bA) \le {2}$. This already fails for the well-understood case when $\Delta(\bA) = 3$ and $n = 2$, which, in light of our previous discussion, could be considered the natural next step towards a possible extension of the method. An instance illustrating this deficiency is given by $C(\bA)$ with 
	\begin{align*}
		\bA = \begin{pmatrix}
			1 & 0 \\ 2 & 3
		\end{pmatrix}
	\end{align*}
and  ${\ve z}=(7, -3)^\top\in C(\bA)\cap\Z^n$. Then $\Delta(\bA) = \det\bA = 3$ and the Hilbert basis elements of $C(\bA)$ is given by the vectors $\be_1,\be_2$, $(2,-1)^\top$, and $(3,-2)^\top$. One can check that $P_{\bm{1}}(\bA)\cap\Z^n = \lbrace\bm{0}\rbrace$ and no Hilbert basis element has the desired property. So already in this comparably simple case the method of reducing to lower-dimensional faces of the cone fails.
	
	
	\bibliographystyle{plain}
	
	\bibliography{references}

\begin{thebibliography}{10}

\bibitem{Aardal_Weismantel_Wolsey}
K.~Aardal, R.~Weismantel, and L.A. Wolsey.
\newblock Non-standard approaches to integer programming.
\newblock volume 123, pages 5--74. 2002.
\newblock Workshop on Discrete Optimization, DO'99 (Piscataway, NJ).

\bibitem{alievAverkovLoeraOertel21}
I.~Aliev, G.~Averkov, J.~de~Loera, and T.~Oertel.
\newblock Sparse representation of vectors in lattices and semigroups.
\newblock {\em Mathematical Programming}, 192, 05 2021.

\bibitem{alievdeloesparelindio2017}
I.~Aliev, {J. De Loera}, T.~Oertel, and C.~O'Neill.
\newblock Sparse solutions of linear diophantine equations.
\newblock {\em SIAM Journal on Applied Algebra and Geometry}, 1:239--253, 2017.

\bibitem{artmannweiszen17}
S.~Artmann, R.~Weismantel, and R.~Zenklusen.
\newblock A strongly polynomial algorithm for bimodular integer linear
  programming.
\newblock In {\em Proceedings of the 49th Annual ACM SIGACT Symposium on Theory
  of Computing}, pages 1206--1219, 2017.

\bibitem{bertsimas-LPbook}
D.~Bertsimas and J.N. Tsitsiklis.
\newblock {\em Introduction to linear optimization}.
\newblock Athena Scientific, 1997.

\bibitem{bonisummaeisenbranddiameterpoly14}
N.~Bonifas, M.~{Di Summa}, F.~Eisenbrand, N.~H{\"a}hnle, and M.~Niemeier.
\newblock On sub-determinants and the diameter of polyhedra.
\newblock {\em Discrete and Computational Geometry}, 52:102--115, 2014.

\bibitem{braun2018detecting}
B.~Braun, R.~Davis, and L.~Solus.
\newblock Detecting the integer decomposition property and ehrhart unimodality
  in reflexive simplices.
\newblock {\em Advances in Applied Mathematics}, 100:122--142, 2018.

\bibitem{brunsgubeladzenormalpoly1999}
W.~Bruns and J.~Gubeladze.
\newblock Normality and covering properties of affine semigroups.
\newblock {\em J. Reine Angew. Mathematik}, 510:151--178, 1999.

\bibitem{brunsgubehenkcounterexampleintcara99}
W.~Bruns, J.~Gubeladze, M.~Henk, A.~Martin, and R.~Weismantel.
\newblock A counterexample to an integer analogue of {C}arath{\'e}odory's
  theorem.
\newblock {\em Journal f{\"u}r reine und angewandte Mathematik}, 510:179--185,
  1999.

\bibitem{celayakuhlpaarweis22}
M.~Celaya, S.~Kuhlmann, J.~Paat, and R.~Weismantel.
\newblock Improving the {C}ook et al. proximity bound given integral valued
  constraints.
\newblock In {\em Integer Programming Combinatorial Optimization: 23rd
  International Conference}, pages 84--97, 2022.

\bibitem{celayakuhlpaatweis2022proxandflatness}
M.~Celaya, S.~Kuhlmann, J.~Paat, and R.~Weismantel.
\newblock Proximity and flatness bounds for integer linear optimization.
\newblock \url{http://arxiv.org/abs/2211.14941}, 2022.

\bibitem{conrads2002weighted}
H.~Conrads.
\newblock Weighted projective spaces and reflexive simplices.
\newblock {\em manuscripta mathematica}, 107(2):215--227, 2002.

\bibitem{CookFS1986}
W.~Cook, J.~Fonlupt, and A.~Schrijver.
\newblock An integer analogue of {C}arath{\'e}odory's theorem.
\newblock {\em Journal of Combinatorial Theory, Series B}, 40(1):63--70, 1986.

\bibitem{pinamatroidcararank2003}
J.C. de~Pina and J.~Soares.
\newblock Improved bound for the {C}arath{\'e}odory rank of the bases of a
  matroid.
\newblock {\em Journal of Combinatorial Theory, Series B}, 88:323--327, 2003.

\bibitem{eisenbrandshmonincaratheodorybounds06}
F.~Eisenbrand and G.~Shmonin.
\newblock Carath{\'e}odory bounds for integer cones.
\newblock {\em Operations Research Letters, 34:564}, 568:2006, 2006.

\bibitem{Fiorini2022IntegerPW}
S.~Fiorini, G.~Joret, S.~Weltge, and Y.~Yuditsky.
\newblock Integer programs with bounded subdeterminants and two nonzeros per
  row.
\newblock {\em 2021 IEEE 62nd Annual Symposium on Foundations of Computer
  Science (FOCS)}, pages 13--24, 2022.

\bibitem{gijswijtipdicp2012}
D.C. Gijswijt and G.~Regts.
\newblock Polyhedra with the {I}nteger {C}arath{\'e}odory property.
\newblock {\em Journal of Combinatorial Theory, Series B}, 102:62--70, 2012.

\bibitem{Graver1975}
J.~E. Graver.
\newblock On the foundations of linear and integer linear programming. {I}.
\newblock {\em Math. Programming}, 9(2):207--226, 1975.

\bibitem{GruLek87}
P.~M. Gruber and C.~G. Lekkerkerker.
\newblock {\em Geometry of numbers}, volume~37 of {\em North-Holland
  Mathematical Library}.
\newblock North-Holland Publishing Co., Amsterdam, second edition, 1987.

\bibitem{gubeladzesurveynormal2023}
J.~Gubeladze.
\newblock Normal polytopes: between discrete, continuous, and random.
\newblock {\em Journal of Pure and Applied Algebra}, 227, 2023.

\bibitem{HKW2022}
M.~Henk, S.~Kuhlmann, and R.~Weismantel.
\newblock On lattice width of lattice-free polyhedra and height of {H}ilbert
  bases.
\newblock {\em SIAM Journal on Discrete Mathematics}, 36(3):1918--1942, 2022.

\bibitem{Jeroslow78}
R.~G. Jeroslow.
\newblock Some basis theorems for integral monoids.
\newblock {\em Math. Oper. Res.}, 3(2):145--154, 1978.

\bibitem{naegelesanzencongruence2022}
M.~N{\"a}gele, R.~Santiago, and R.~Zenklusen.
\newblock Congruency-constrained {TU} problems beyond the bimodular case.
\newblock In {\em Proceedings of the 2022 {ACM-SIAM} Symposium on Discrete
  Algorithms, {SODA}}, pages 2743--2790. {SIAM}, 2022.

\bibitem{Schrijver81}
A.~Schrijver.
\newblock On total dual integrality.
\newblock {\em Linear Algebra Appl.}, 38:27--32, 1981.

\bibitem{schrijvertheorylinint86}
A.~Schrijver.
\newblock {\em Theory of Linear and Integer Programming}.
\newblock Wiley, 1986.

\bibitem{sebohilbertbasisdreidim90}
A~Seb{\H o}.
\newblock {H}ilbert bases, {C}arath{\'e}odory's theorem and combinatorial
  optimization.
\newblock {\em Proceedings of the 1st Integer Programming and Combinatorial
  Optimization Conference}, 1990.

\bibitem{vandercorputhilberbasisunique1931}
J.~van~der Corput.
\newblock Konstruktion der {M}inimalbasis f{\"u}r spezielle {D}iophantische
  {S}ysteme von linear-homogenen {G}leichungen und {U}ngleichungen.
\newblock {\em Proc. Roy. Acad.}, 34:515--523, 1931.

\bibitem{veselovchirkovbimodular09}
S.I. Veselov and A.J. Chirkov.
\newblock Integer program with bimodular matrix.
\newblock {\em Discrete Optimization}, 6:220--222, 2009.

\end{thebibliography}

\end{document}